\documentclass[12pt,a4 paper]{article}
\usepackage{epsfig}
\usepackage{amsmath}
\usepackage{amsfonts}
\usepackage{amsmath,amscd}
\usepackage{fullpage}
\usepackage{amssymb}
\usepackage{amsthm}
\usepackage{mathrsfs}
\usepackage{graphicx}
\usepackage{xypic}
\newtheorem{definition}{Definition}[section]
\newtheorem{theorem}[definition]{Theorem}
\newtheorem{remark}[definition]{Remark}
\newtheorem{final Remarks}[definition]{Final Remarks}

\newtheorem{corollary}[definition]{Corollary}

\numberwithin{equation}{section}
\begin{document}
\title{Inertia groups and smooth structures of $(n-1)$-connected $2n$-manifolds}
\vspace{2cm}
\author{Ramesh Kasilingam}
\date{}
\maketitle
\begin{abstract}
Let $M^{2n}$ denote a closed $(n-1)$-connected smoothable topological $2n$-manifold. We show that the group $\mathcal{C}(M^{2n})$ of concordance classes of smoothings of $M^{2n}$ is isomorphic to the group of smooth homotopy spheres $\overline{\Theta}_{2n}$ for $n=4$ or $5$, the concordance inertia group $I_c(M^{2n})=0$ for $n=3$, $4$, $5$ or $11$ and the homotopy inertia group $I_h(M^{2n})=0$ for $n=4$. On the way, following Wall's approach \cite{Wal67} we present a new proof of the main result in \cite{KS07}, namely, for $n=4$, $8$ and $H^{n}(M^{2n};\mathbb{Z})\cong \mathbb{Z}$, the inertia group $I(M^{2n})\cong \mathbb{Z}_2$. We also show that, up to orientation-preserving diffeomorphism, $M^{8}$ has at most two distinct smooth structures; $M^{10}$ has exactly six distinct smooth structures and then show that if $M^{14}$ is a $\pi$-manifold, $M^{14}$ has exactly two distinct smooth structures. 
\end{abstract}
\paragraph{Classification.}
57R55; 57R60; 57R50; 57R65.
\paragraph{Acknowledgments.}
The author would like to thank his advisor, Prof. A. R. Shastri for several helpful suggestions and questions. 
\section{\large Introduction}
\label{intro}
We work in the categories of closed, oriented, simply-connected $Cat$-manifolds $M$ and $N$ and orientation preserving maps, where $Cat=Diff$ for smooth manifolds or $Cat=Top$ for topological manifolds. Let $\overline{\Theta}_m$ be the group of smooth homotopy spheres defined by M. Kervaire and J. Milnor in \cite{KM63}. Recall that the collection of homotopy spheres $\Sigma$  which admit a diffeomorphism $M\to M\#\Sigma$ form a subgroup $I(M)$ of $\overline{\Theta}_m$, called the inertia group of $M$, where we regard the connected sum $M\#\Sigma^m$ as a smooth manifold with the same underlying topological space as $M$ and with smooth structure differing from that of $M$ only on an $n$-disc. The homotopy inertia group $I_h(M)$ of $M^m$ is a subset of the inertia group consisting of homotopy spheres $\Sigma$ for which the identity map $\rm{id}:M\to M\#\Sigma^m$ is homotopic to a diffeomorphism. Similarly, the concordance inertia group of $M^m$, $I_c(M^m)\subseteq \overline{\Theta}_m$, consists of those homotopy spheres $\Sigma^m$ such that $M$ and $M\#\Sigma^m$ are concordant.\\
\indent The paper is organized as following. Let $M^{2n}$ denote a closed $(n-1)$-connected smoothable topological $2n$-manifold. In section 2, we show that the group $\mathcal{C}(M^{2n})$ of concordance classes of smoothings of $M^{2n}$ is isomorphic to the group of smooth homotopy spheres $\overline{\Theta}_{2n}$ for $n=4$ or $5$, the concordance inertia group $I_c(M^{2n})=0$ for $n=3$, $4$, $5$ or $11$ and the homotopy inertia group $I_h(M^{2n})=0$ for $n=4$.\\
\indent In section 3, we present a new proof of the following result in \cite{KS07}. 
\begin{theorem}\label{iner0}
Let $M^{2n}$ be an $(n-1)$-connected closed smooth manifold of dimension $2n\neq4$ such that  $H^{n}(M;\mathbb{Z})\cong \mathbb{Z}$. Then the inertia group  $I(M^{2n})\cong \mathbb{Z}_2$.
\end{theorem}
\indent In section 4, we show that, up to orientation-preserving diffeomorphism, $M^{8}$ has at most two distinct smooth structures; $M^{10}$ has exactly six distinct smooth structures and if $M^{14}$ is a $\pi$-manifold, then $M^{14}$ has exactly two distinct smooth structures. 
\section{Concordance inertia groups of $(n-1)$-connected $2n$-manifolds}
We recall some terminology from \cite{KM63}:
\begin{definition}\rm
\begin{itemize}
\item[(a)] A homotopy $m$-sphere $\Sigma^m$ is a closed oriented smooth manifold homotopy equivalent to the standard unit sphere $\mathbb{S}^m$ in $\mathbb{R}^{m+1}$.
\item[(b)]A homotopy $m$-sphere $\Sigma^m$ is said to be exotic if it is not diffeomorphic to $\mathbb{S}^m$.
\end{itemize}
\end{definition}
\begin{definition}\rm
Define the $m$-th group of smooth homotopy spheres $\Theta_m$ as follows. Elements are oriented $h$-cobordism classes $[\Sigma]$ of homotopy $m$-spheres $\Sigma$, where $\Sigma$ and $\Sigma^{\prime}$ are called (oriented) $h$-cobordant if there is an oriented $h$-cobordism $(W, \partial_0W,\partial_1W)$ together with orientation preserving diffeomorphisms $\Sigma\to \partial_0W$ and $(\Sigma^{\prime})^-\to \partial_1W$. The addition is given by the connected sum. The zero element is represented by $\mathbb{S}^m$. The inverse of $[\Sigma]$ is given by $[\Sigma^-]$, where $\Sigma^-$ is obtained from $\Sigma$ by reversing the orientation.  M. Kervaire and J. Milnor \cite{KM63} showed that each $\Theta_m$ is a finite abelian group $(m\geq1)$.
\end{definition}
\begin{definition}\rm
Two homotopy $m$-spheres $\Sigma^{m}_{1}$ and $\Sigma^{m}_{2}$ are said to be equivalent if there exists an orientation preserving diffeomorphism $f:\Sigma^{m}_{1}\to \Sigma^{m}_{2}$.\\
\indent The set of equivalence classes of homotopy $m$-spheres is denoted by $\overline{\Theta}_m$. The Kervaire-Milnor \cite{KM63} paper worked rather with the group $\Theta_m$ of smooth homotopy spheres up to $h$-cobordism. This makes a difference only for $m=4$, since it is known, using the $h$-cobordism theorem of Smale \cite{Sma62}, that $\Theta_m\cong \overline{\Theta}_m$ for $m\neq 4$. However the difference is important in the four dimensional case, since $\Theta_4$ is trivial, while the structure of $\overline{\Theta}_4$ is a great unsolved problem.
\end{definition}
\begin{definition}\rm
Let $M$ be a closed topological manifold. Let $(N,f)$ be a pair consisting of a smooth manifold $N$ together with a homeomorphism $f:N\to M$. Two such pairs $(N_{1},f_{1})$ and $(N_{2},f_{2})$ are concordant provided there exists a diffeomorphism $g:N_{1}\to N_{2}$ such that the composition $f_{2}\circ g$ is topologically concordant to $f_{1}$, i.e., there exists a homeomorphism $F: N_{1}\times [0,1]\to M\times [0,1]$ such that $F_{|N_{1}\times 0}=f_{1}$ and $F_{|N_{1}\times 1}=f_{2}\circ g$. The set of all such concordance classes is denoted by $\mathcal{C}(M)$.\\
\indent We will denote the class in $\mathcal{C}(M)$ of $(M^n\#\Sigma^n, \rm{id})$ by $[M^n\#\Sigma^n]$. (Note that $[M^n\#\mathbb{S}^n]$ is the class of $(M^n, \rm{id})$.)
\end{definition}
\begin{definition}\rm{
Let $M^m$ be a closed smooth $m$-dimensional manifold. The inertia group $I(M)\subset \overline{\Theta}_{m}$ is defined as the set of $\Sigma \in \overline{\Theta}_{m}$ for which there exists a diffeomorphism $\phi :M\to M\#\Sigma$.\\
\indent Define the homotopy inertia group $I_h(M)$ to be the set of all $\Sigma\in I(M)$ such that there exists a diffeomorphism $M\to M\#\Sigma$ which is homotopic to $\rm{id}:M \to M\#\Sigma$.\\
\indent Define the concordance inertia group $I_c(M)$ to be the set of all $\Sigma\in I_h(M)$ such that $M\#\Sigma$ is concordant to $M$.}
\end{definition}
\begin{remark}\rm
\indent
\begin{itemize}
\item[(1)] Clearly, $I_c(M)\subseteq I_h(M)\subseteq I(M)$.
\item [(2)] For $M=\mathbb{S}^m$, $I_c(M)=I_h(M)=I(M)=0$.
\end{itemize}
\end{remark}
Now we have the following:
\begin{theorem}\label{conwall}
Let $M^{2n}$ be a closed smooth $(n-1)$-connected ${2n}$-manifold with $n\geq 3$. 
\begin{itemize}
 \item [{\rm(i)}]If $n$ is any integer such that $\Theta_{n+1}$ is trivial, then $I_c(M^{2n})=0$.
 \item [{\rm(ii)}]If $n$ is any integer greater than $3$ such that $\Theta_{n}$ and $\Theta_{n+1}$ are trivial, then $$\mathcal{C}(M^{2n})=\left \{ [M^{2n}\#\Sigma]~~ |~~\Sigma\in \overline{\Theta}_{2n} \right \}\cong \overline{\Theta}_{2n}.$$
 \item[{\rm(iii)}] If $n=8$ and $H^n(M;\mathbb{Z})\cong \mathbb{Z}$, then $M^{2n}\#\Sigma^{2n}$ is not concordant to $M^{2n}$, where $\Sigma^{2n}\in \overline{\Theta}_{2n}$ is the exotic sphere. In particular, $\mathcal{C}(M^{2n})$ has at least two elements.
  \item[{\rm (iv)}]If $n$ is any even integer such that $\Theta_{n}$ and $\Theta_{n+1}$ are trivial, then $I_h(M)=0$.
\end{itemize}
\end{theorem}
\begin{proof}
Let $Cat=Top~~ {\rm{or}}~~ G$, where $Top~~ {\rm{and}}~~ G$ are the stable spaces of self homeomorphisms of $\mathbb{R}^{n}$ and self homotopy equivalences  of $\mathbb{S}^{n-1}$ respectively. For any degree one map $f_{M}:M\to \mathbb{S}^{2n}$, we have a homomorphism
$$f_{M}^*:[\mathbb{S}^{2n}, Cat/O]\to [M ,Cat/O].$$ By Wall \cite {Wal62}, $M$ has the homotopy type of $X=(\bigvee_{i=1}^{k} \mathbb{S}_{i}^n)\bigcup_{g}\mathbb{D}^{2n}$, where $k$ is the $n$-th Betti number of $M$, $\bigvee_{i=1}^{k} \mathbb{S}_{i}^n$ is the wedge sum of $n$-spheres and $g:\mathbb{S}^{2n-1}\to \bigvee_{i=1}^{k} \mathbb{S}_{i}^n$ is the attaching map of $\mathbb{D}^{2n}$. Let $\phi:M\to X$ be a homotopy equivalence of degree one and $q:X\to \mathbb{S}^{2n}$ be the collapsing map obtained by identifying $\mathbb{S}^{2n}$ with $X/\bigvee_{i=1}^{k} \mathbb{S}_{i}^n$ in an orientation preserving way. Let $f_{M}=q\circ \phi:M\to \mathbb{S}^{2n}$ be the degree one map.\\
\indent Consider the following Puppe's exact sequence for the inclusion $i:\bigvee_{i=1}^{k} \mathbb{S}_{i}^n \hookrightarrow X$ along $Cat/O$:
\begin{equation}\label{longG}
....\longrightarrow [\bigvee_{i=1}^{k}S \mathbb{S}_{i}^n, Cat/O]\stackrel{(S(g))^{*}}{\longrightarrow}[\mathbb{S}^{2n}, Cat/O]\stackrel{q^{*}}{\longrightarrow}[X, Cat/O]\stackrel{i^{*}}{\longrightarrow}[\bigvee_{i=1}^{k} \mathbb{S}_{i}^n, Cat/O],
\end{equation}
where $S(g)$ is the suspension of the map $g:\mathbb{S}^{2n-1}\to \bigvee_{i=1}^{k} \mathbb{S}_{i}^n$.\\
\indent Using the fact that $$[\bigvee_{i=1}^{k}S \mathbb{S}_{i}^n, Cat/O]\cong \prod_{i=1}^{k}[\mathbb{S}_{i}^{n+1}, Cat/O]$$ and $$[\bigvee_{i=1}^{k} \mathbb{S}_{i}^n, Cat/O]\cong \prod_{i=1}^{k}[\mathbb{S}_{i}^n, Cat/O],$$ the above exact sequence (\ref{longG}) becomes $$....\longrightarrow \prod_{i=1}^{k}[\mathbb{S}_{i}^{n+1}, Cat/O] \stackrel{(S(g))^{*}}{\longrightarrow}[\mathbb{S}^{2n}, Cat/O] \stackrel{q^{*}}{\longrightarrow}[X, Cat/O]\stackrel{i^{*}}{\longrightarrow} \prod_{i=1}^{k}[\mathbb{S}_{i}^n, Cat/O].$$\\
(i):
If $n$ is any integer such that $\Theta_{n+1}$ is trivial and $Cat=Top$ in the above exact sequence (\ref{longG}), by using the fact that $$[\mathbb{S}^{m}, Top/O]=\overline{\Theta}_{m} ~(m\neq 3, 4)$$ and $[\mathbb{S}^{4}, Top/O]=0$ (\cite[pp. 200-201]{KS77}), we have $q^{*}:[\mathbb{S}^{2n}, Top/O]\to [X, Top/O]$ is injective. Hence $f_{M}^*=\phi^{*}\circ q^{*}:\overline{\Theta}_{2n}\to [M, Top/O]$ is injective. By using the identifications $\mathcal{C}(M^{2n})=[M,Top/O]$ given by \cite[pp. 194-196]{KS77},  $f_{M}^*:\overline{\Theta}_{2n}\to \mathcal{C}(M^{2n})$ becomes $[\Sigma^{2n}]\to [M\#\Sigma^{2n}]$. $I_c(M)$ is exactly the kernel of $f_{M}^{*}$, and so $I_c(M)=0$. This proves (i).\\
(ii): If $n>3$, $\Theta_{n}$ and $\Theta_{n+1}$ are trivial, and $Cat=Top$ then, from the above exact sequence (\ref{longG}) we have $q^{*}:[\mathbb{S}^{2n}, Top/O]\to [X, Top/O]$ is an isomorphism. This shows that $f_{M}^*=\phi^{*}\circ q^{*}:\overline{\Theta}_{2n}\to \mathcal{C}(M^{2n})$ is an isomorphism and hence $$\mathcal{C}(M^{2n})=\{ [M^{2n}\#\Sigma] ~~|~~ \Sigma\in \overline{\Theta}_{2n} \}.$$ This proves (ii).\\
(iii): If $n=8$ and $H^n(M;\mathbb{Z})\cong \mathbb{Z}$, then $M^{2n}$ has the homotopy type of $X=\mathbb{S}^{n}\bigcup_{g}\mathbb{D}^{2n}$, where $g:\mathbb{S}^{2n-1}\to \mathbb{S}^{n}$ is the attaching map. In order to prove $M^{2n}\#\Sigma^{2n}$ is not concordant to $M^{2n}$, by the above exact sequence (\ref{longG}) for $Cat=Top$, it suffices to prove $q^{*}:[\mathbb{S}^{16}, Top/O]\to [X, Top/O]$ is monic, which is equivalent to saying that $(S(g))^{*}:[S \mathbb{S}^8, Top/O]\to [\mathbb{S}^{16}, Top/O]$ is the zero homomorphism. For the case $g=p$, where $p:\mathbb{S}^{15}\to \mathbb{S}^8$ is the Hopf map, $(S(g))^{*}$ is the zero homomorphism, which was proved in the course of proof of lemma 1 in \cite[pp. 58-59]{AF03}. This proof works verbatim for any map $g:\mathbb{S}^{2n-1}\to \mathbb{S}^{n}$ as well. This proves (iii).\\
(iv): If $n$ is any even integer such that  $\Theta_{n}$ and $\Theta_{n+1}$ are trivial, then $\pi_{n+1}(G/O)=0$. This shows that from the above exact sequence (\ref{longG}) for $Cat=G$,  $q^{*}:[\mathbb{S}^{2n}, G/O]\to [X ,G/O]$ is injective. Then $f_{M}^*=\phi^{*}\circ q^{*}:[\mathbb{S}^{2n}, G/O]\to [M,G/O]$ is injective. From the surgery exact sequences of $M$ and $\mathbb{S}^{2n}$, we get the following commutative diagram (\cite[Lemma 3.4]{Cro10}):
\begin{equation}\label{digram1}
\begin{CD}
L_{2n+1}(e)@>>> \overline{\Theta}_{2n} @>\eta_{\mathbb{S}^{2n}}>> \pi_{2n}(G/O)@>>>  L_{2n}(e)\\
 @VV=V            @VVf_{M}^{\bullet}V             @VVf_{M}^*V                         @VV=V\\
L_{2n+1}(e)@>>> \mathcal{S}^{Diff}(M)  @>\eta_{M}>>    [M, G/O]     @>>>   L_{2n}(e)
\end{CD}
\end{equation}
By using the facts that $L_{2n+1}(e)=0$,  injectivity of $\eta_{\mathbb{S}^{2n}}$ and $\eta_{M}$ follow from the diagram, and combine with the injectivity of $f_{M}^{*}$ to show that $f_{M}^{\bullet}: \overline{\Theta}_{2n}\to \mathcal{S}^{Diff}(M)$ is injective. $I_h(M)$ is exactly the kernel of $f_{M}^{\bullet}$, and so $I_h(M)=0$. This proves (iv).
\end{proof}
\begin{remark}\label{homoiner}\rm
\indent
\begin{itemize}
\item[(i)] By M. Kervaire and J. Milnor \cite{KM63}, $\Theta_m=0$ for $m=1$, $2$, $3$, $4$, $5$, $6$ or $12$.
If $M^{2n}$ is a closed smooth $(n-1)$-connected $2n$-manifold, by Theorem \ref{conwall}(i) and (ii), $I_c(M^{2n})=0$ for $n=3$, $4$, $5$ or $11$ and $\mathcal{C}(M^{2n})\cong \overline{\Theta}_{2n}$ for $n=4$ or 5.
\item[(ii)] If $M$ has the homotopy type of $\mathbb{O}\textbf{P}^2$, by Theorem \ref{iner0} and Theorem \ref{conwall}(iii), we have $I_c(M)=0\neq I(M)$.
\item[(iii)] By Theorem \ref{conwall}(iv), if $M$ has the homotopy type of $\mathbb{H}\textbf{P}^2$, then $I_h(M)=0$.
\end{itemize}
\end{remark}
\begin{definition}\rm
Let $M$ and $N$ are smooth manifolds. A smooth map $f: M\to N$ is called tangential if for some integers $k$, $l$, $f^{*}(T(N))\oplus \displaystyle{\epsilon}_{M}^{k}\cong T(M)\oplus \displaystyle{\epsilon}_{M}^{l}$.
\end{definition}
\begin{definition}\rm
Let $M$ be a topological manifold. Let $(N,f)$ be a pair consisting of a smooth manifold $N$ together with a tangential homotopy equivalence of degree one $f:N\to M$. Two such pairs $(N_{1},f_{1})$ and $(N_{2},f_{2})$ are equivalent provided there exists a diffeomorphism $g:N_{1}\to N_{2}$ such that $f_{2}\circ g$ is homotopic to $f_{1}$. The set of all such equivalence classes is denoted by $\theta(M)$.
\end{definition}
For $M=\mathbb{H}\textbf{P}^2$, \cite[Theorem 4]{Her69} shows $\theta(\mathbb{H}\textbf{P}^2)$ contains at most two elements. Now by Remark \ref{homoiner}(iii), we have the following:
\begin{corollary}
{\rm $\theta(\mathbb{H}\textbf{P}^2)$} contains exactly two elements, with representatives given by {\rm $(\mathbb{H}\textbf{P}^2, \rm{id})$} and {\rm $(\mathbb{H}\textbf{P}^2\#\Sigma^{8}, \rm{id})$}, where $\Sigma^{8}$ is the exotic $8$-sphere.
\end{corollary}
\section{Inertia groups of projective plane-like manifolds}
In \cite{Wal62}, C.T.C. Wall assigned to each closed oriented $(n-1)$-connected $2n$-dimensional smooth manifold $M^{2n}$ with $n\geq 3$, a system of invariants as follows:
 \begin{itemize}
  \item [(1)] $H=H^n(M; \mathbb{Z})\cong {\rm{Hom}}(H_n(M; \mathbb{Z}), \mathbb{Z})\cong \oplus_{j=1}^{k} \mathbb{Z}$, the cohomology group of $M$, with $k$ the $n$-th Betti number of $M$,
  \item [(2)] $I : H\times H\to \mathbb{Z}$, the intersection form of $M$ which is unimodular and $n$-symmetric, defined by
$$I(x,y)=\left \langle x\cup y, [M] \right \rangle,$$ where the homology class $[M]$ is the orientation class of $M$,
\item [(3)] A map $\alpha: H^n(M; \mathbb{Z})\to \pi_{n-1}(SO_n)$ that assigns each element $x\in H^n(M; \mathbb{Z})$ to the characteristic map $\alpha(x)$ for the normal bundle of the embedded $n$-sphere $\mathbb{S}^{n}_x$ representing $x$. 
\end{itemize}
Denote by $\chi = S\circ \alpha: H^n(M; \mathbb{Z})\to \pi_{n-1}(SO_{n+1})\cong \widetilde{KO}(\mathbb{S}^n)$, where
$S:\pi_{n-1}(SO_n)\to \pi_{n-1}(SO_{n+1})$ is the suspension map. Then $$\chi = S\circ ~ \alpha \in H^n(M; \widetilde{KO}(\mathbb{S}^n))= {\rm{Hom}}(H^n(M; \mathbb{Z});\widetilde{KO}(\mathbb{S}^n))$$
can be viewed as an $n$-dimensional cohomology class of $M$, with coefficients in $\widetilde{KO}(\mathbb{S}^n)$. The obstruction to triviality of the tangent bundle over the $n$-skeleton is the element  $\chi\in H^n(M; \widetilde{KO}(\mathbb{S}^n))$ \cite{Wal62}. By \cite[pp. 179-180]{Wal62}, the Pontrjagin class of $M^{2n}$ is given by 
\begin{equation}\label{pontr1}
\begin{split}
p_m(M^{2n})&=\pm a_{m}(2m-1)! \chi,
\end{split}
\end{equation}
where $n=4m$ and 
\begin{eqnarray*}
a_{m} =  \left\{
\begin{array}{l}
1~~ if\ \  4m\equiv 0 ~~\rm{(mod~~8)}. \\ \\

2~~ if\ \  4m\equiv 4 ~~\rm{(mod~~8)}.
\end{array}
\right .
\end{eqnarray*} 
\indent Define $\Theta_n(k)$ to be the subgroup of $\overline{\Theta}_n$ consisting of those homotopy $n$-sphere $\Sigma^n$ which are the boundaries of $k$-connected $(n+1)$-dimensional compact manifolds, $1\leq k< [n/2]$. Thus, $\Theta_n(k)$ is the kernel of the natural map $i_k:\overline{\Theta}_n\to \Omega_n(k)$, where  $\Omega_n(k)$ is the $n$-dimensional group in $k$-connective cobordism theory \cite{Sto68} and $i_k$ sends $\Sigma^n$ to its cobordism class. Using surgery, we see $\Omega_{*}(1)$ is the usual oriented cobordism group. So $\overline{\Theta}_n=\Theta_n(1)$. Similarly, $\Omega_n(2)\cong \Omega_n^{Spin}$ $(n\geq 7)$; since $BSpin$ is, in fact, 3-connected, for $n\geq 8$,  $\Omega_n(2)\cong \Omega_n(3)$ and $\Theta_n(2)=\Theta_n(3)=bSpin_n$. Here $bSpin_n$ consists of homotopy $n$-sphere which bound spin manifolds.\\
\indent In \cite{Wal67}, C.T.C. Wall defined the Grothendieck group $\mathcal{G}^{2n+1}_{n}$, a homomorphism $\vartheta:\mathcal{G}^{2n+1}_{n}\to \overline{\Theta}_{2n}$ such that $\vartheta(\mathcal{G}^{2n+1}_{n})=\Theta_{2n}(n-1)$ and proved the following theorem :
\begin{theorem}$(Wall)$\label{wall}
 Let $M^{2n}$ be a $(n-1)$-connected $2n$-manifold and $\Sigma^{2n}$ be a homotopy sphere in $\overline{\Theta}_{2n}$. Then 
 $M\#\Sigma^{2n}$ is an orientation-preserving diffeomorphic to $M$ if and only if 
 \begin{itemize}
  \item [\rm{(i)}] $\Sigma^{2n}=0$ in $\overline{\Theta}_{2n}$ or 
  \item [\rm{(ii)}] $\chi \not \equiv 0~~ \rm{(mod~~2)}$ and $\Sigma^{2n}\in \vartheta(\mathcal{G}^{2n+1}_{n})=\Theta_{2n}(n-1)$
 \end{itemize} 
\end{theorem}
We also need the following result from \cite{ABP67} :
\begin{theorem}$(Anderson, Brown, Peterson)$\label{ABP67}
 Let $\eta_n:\overline{\Theta}_n\to \Omega_n^{Spin}$ be the homomorphism such that $\eta_n$ sends $\Sigma^n$ to its spin cobordism class. Then
 $\eta_n\neq 0$ if and only if $n=8k+1$ or $8k+2$.
\end{theorem}
{\bf{Proof of Theorem \ref{iner0}:}}
Let $\xi$ be a generator of $H^{n}(M^{2n}; \mathbb{Z})$. Consider the case $n=4$. Then by Itiro Tamura \cite{Tam61} and (\ref{pontr1}), the Pontrjagin class of $M^{2n}$ is given by
\begin{align*}
p_1(M^{2n})&=2(2h+1)\xi=\pm 2\chi,
\end{align*}
where $h\in \mathbb{Z}$. This implies that $$\chi=\pm (2h+1)\xi.$$ Likewise, for $n=8$, we have
\begin{align*}
p_2(M^{2n})&=6(2k+1)\xi=\pm 6\chi,
\end{align*}
where $k\in \mathbb{Z}$. This implies that $$\chi=\pm (2k+1)\xi.$$ Therefore in either case, $\chi\not \equiv 0 ~\rm{(mod~2)}$. Now by Theorem \ref{wall}, it follows that $$I(M^{2n})=\Theta_{2n}(n-1).$$ Since $\Theta_{2n}(n-1)$ is the kernel of the natural map $i_{n-1}:\overline{\Theta}_{2n}\to \Omega_{2n}(n-1)$, where $\Omega_{2n}(n-1)\cong \Omega_{8}^{Spin}$ for $n=4$ and $\Omega_{2n}(n-1)\cong \Omega_{16}^{String}\cong \mathbb{Z}\oplus \mathbb{Z}$ for $n=8$ \cite{Gia71}. Now by Theorem \ref{ABP67} and using the fact that $\overline{\Theta}_{16}\cong \mathbb{Z}_2$ \cite{KM63}, we have $i_{n-1}=0$ for $n=4$ and $8$. This shows that $\Theta_{2n}(n-1)=\overline{\Theta}_{2n}$. This implies that $$I(M^{2n})\cong \mathbb{Z}_2.$$ This completes the proof of Theorem \ref{iner0}.
\section{Smooth structures of $(n-1)$-connected $2n$-manifolds}
\begin{definition}\label{smoo.stru}($Cat=Diff~~{\rm{or}}~~ Top$-structure sets)\cite{Cro10}~\rm
 Let $M$ be a closed $Cat$-manifold. We define the $Cat$-structure set $\mathcal{S}^{Cat}(M)$ to be the set of equivalence classes of pairs $(N,f)$ where $N$ is a closed $Cat$-manifold and $f: N\to M$ is a homotopy equivalence.
And the equivalence relation is defined as follows :
\begin{center}
  $(N_1,f_1)\sim  (N_2,f_2)$ if there is a $Cat$-isomorphism $\phi:N_1\to N_2$ \\
  such that $f_2\circ h$ is homotopic to $f_1$.
\end{center}
We will denote the class in $\mathcal{S}^{Cat}(M)$ of $(N,f)$ by $[(N,f)]$. The base point of $S^{Cat}(M)$ is the equivalence class $[(M, \rm{id})]$ of $\rm{id} : M\to M$.
 \end{definition}
The forgetful maps $F_{Diff} : \mathcal{S}^{Diff}(M)\to \mathcal{S}^{Top}(M)$ and $F_{Con} : \mathcal{C}(M)\to \mathcal{S}^{Diff}(M)$ fit into a short exact sequence of pointed sets \cite{Cro10}:
$$\mathcal{C}(M)\stackrel{F_{Con}} {\longrightarrow} \mathcal{S}^{Diff}(M)\stackrel{F_{Diff}} {\longrightarrow} \mathcal{S}^{Top}(M).$$
\begin{theorem}\label{class}
Let $n$ be any integer greater than $3$ such that $\Theta_{n}$ and $\Theta_{n+1}$ are trivial and $M^{2n}$ be a closed smooth $(n-1)$-connected $2n$-manifold. Let $f:N\to M$ be a homeomorphism where $N$ is a closed smooth manifold. Then
\begin{itemize}
\item[(i)]there exists a diffeomorphism $\phi:N\to M\#\Sigma^{2n}$, where $\Sigma^{2n}\in \overline{\Theta}_{2n}$ such that the following diagram commutes up to homotopy:
$$\xymatrix{
N \ar[r]^\phi \ar[rd]^f &M\#\Sigma^{2n} \ar[d]^{\rm{id}}\\
&M}$$
\item[(ii)] If $I_h(M)=\overline{\Theta}_{2n}$, then $f:N\to M$ is homotopic to a diffeomorphism.
\end{itemize}
\end{theorem}
\begin{proof}
Consider the short exact sequence of pointed sets $$\mathcal{C}(M)\stackrel{F_{Con}} {\longrightarrow}\mathcal{S}^{Diff}(M)\stackrel{F_{Diff}} {\longrightarrow} \mathcal{S}^{Top}(M).$$  By Theorem \ref{conwall}(ii), we have 
$$\mathcal{C}(M)= \left \{[M\#\Sigma]~~|~~\Sigma\in \overline{\Theta}_{2n} \right \}\cong \overline{\Theta}_{2n}.$$
Since  $[(N,f)]\in F_{Diff}^{-1}([(M, \rm{id})])$, we obtain  $$[(N,f)]\in {\rm{Im}}(F_{Con})=\left \{[M\#\Sigma]~~|~~\Sigma\in \overline{\Theta}_{2n} \right \}.$$ This implies that there exists  a homotopy sphere $\Sigma^{2n}\in \overline{\Theta}_{2n}$ such that $(N,f)\sim (M\#\Sigma^{2n}, \rm{id})$ in $\mathcal{S}^{Diff}(M)$. This implies that there exists a diffeomorphism $\phi:N\to M\#\Sigma^{2n}$ such that $f$ is homotopic to $\rm{id}\circ \phi$. This proves (i).\\
\indent If $I_h(M)=\overline{\Theta}_{2n}$, then ${\rm{Im}}(F_{Con})=\{[(M, \rm{id})]\}$ and hence $(N, f)\sim (M, \rm{id})$ in $\mathcal{S}^{Diff}(M)$. This shows that $f:N\to M$ is homotopic to a diffeomorphism $N\to M$. This proves (ii).
\end{proof}
\begin{theorem}\label{atmosttwo}
Let $n$ be any integer greater than $3$ such that $\Theta_{n}$ and $\Theta_{n+1}$ are trivial and $M^{2n}$ be a closed smooth $(n-1)$-connected $2n$-manifold. Then 
the number of distinct smooth structures on $M^{2n}$ up to diffeomorphism is less than or equal to the cardinality of $\overline{\Theta}_{2n}$. In particular, the set of diffeomorphism classes of smooth structures on $M^{2n}$ is $\left \{[M\#\Sigma]~~|~~ \Sigma\in \overline{\Theta}_{2n} \right \}$. 
\end{theorem}
\begin{proof}
By Theorem \ref{class}(i), if $N$ is a closed smooth manifold homeomorphic to $M$, then $N$ is diffeomorphic to $M\#\Sigma^{2n}$ for some homotopy $2n$-sphere $\Sigma^{2n}$. This implies that the set of diffeomorphism classes of smooth structures on $M^{2n}$ is $\left \{[M\#\Sigma]~~|~~ \Sigma\in \overline{\Theta}_{2n} \right \}$. This shows that the number of distinct smooth structures on $M^{2n}$ up to diffeomorphism is less than or equal to the cardinality of $\overline{\Theta}_{2n}$.
\end{proof}
\begin{remark}\rm
\indent
\begin{itemize}
\item[(1)] By Theorem \ref{atmosttwo}, every closed  smooth 3-connected 8-manifold has at most two distinct smooth structures up to diffeomorphism.
 \item[(2)]If  $M^{8}$ is a closed smooth 3-connected 8-manifold such that  $H^{4}(M;\mathbb{Z})\cong \mathbb{Z}$, then by Theorem \ref{iner0}, $I(M)\cong \mathbb{Z}_2$. Now by Theorem \ref{atmosttwo}, $M$ has a unique smooth structure up to diffeomorphism.
 \item[(3)] If $M=\mathbb{S}^4\times \mathbb{S}^4$, then by Theorem \ref{atmosttwo}, $\mathbb{S}^4\times \mathbb{S}^4$ has at most two distinct smooth structures up to diffeomorphism, namely, $\left \{\textbf{[}\mathbb{S}^4\times \mathbb{S}^4\textbf{]}, \textbf{[}\mathbb{S}^4\times \mathbb{S}^4\#\Sigma\textbf{]}\right \}$, where $\Sigma$ is the exotic 8-sphere. However, by \cite[Theorem A]{Sch71}, $I(\mathbb{S}^4\times \mathbb{S}^4)=0$. This implies that $\mathbb{S}^4\times \mathbb{S}^4$ has exactly two distinct smooth structures.
 \end{itemize}
\end{remark}
\begin{theorem}
 Let $M$ be a closed smooth $3$-connected $8$-manifold with stable tangential invariant $\chi=S\alpha:H_4(M)\to \pi_3(SO)=\mathbb{Z}$. Then $M$ has exactly two distinct smooth structures up to diffeomorphism if and only if ${\rm{Im}}(S\alpha)\subseteq 2\mathbb{Z}$.
\end{theorem}
\begin{proof}
Suppose $M$ has exactly two distinct smooth structures up to diffeomorphism. Then by Theorem \ref{atmosttwo}, $M$ and $M\#\Sigma$ are not diffeomorphic, where $\Sigma$ is the exotic 8-sphere. Since $\overline{\Theta}_8=\Theta_8(3)$, by Theorem \ref{wall}, the stable tangential invariant $\chi$ is zero ${\rm(mod~2)}$ and hence ${\rm{Im}}(S\alpha)\subseteq 2\mathbb{Z}$. Conversely, suppose ${\rm{Im}}(S\alpha)\subseteq 2\mathbb{Z}$. Now by Theorem \ref{wall}, $M$ can not be diffeomorphic to $M\#\Sigma$, where $\Sigma$ is the exotic 8-sphere. Now by Theorem \ref{atmosttwo}, $M$ has exactly two distinct smooth structures up to diffeomorphism.
\end{proof}
\begin{remark}\rm
If $n=2$, $3$, $5$, $6$, $7$ ${\rm{(mod~ 8)}}$ or the stable tangential invariant $\chi$ of $M^{2n}$ is zero ${\rm{(mod~ 2)}}$, then by \cite[Corollary, pp. 289]{Wal67} and Theorem \ref{wall}, we have $I(M^{2n})=0$. So, by Theorem \ref{atmosttwo}, we have the following:
\end{remark}
\begin{theorem}\label{tangent}
 Let $n$ be any integer greater than $3$ such that $\Theta_{n}$ and $\Theta_{n+1}$ are trivial and $M^{2n}$ be a closed smooth $(n-1)$-connected $2n$-manifold. If $n=$2, 3, 5, 6, 7 {\rm{(mod 8)}} or the stable tangential invariant $\chi$ of $M^{2n}$ is zero ${\rm{(mod~ 2)}}$, then the set of diffeomorphism classes of smooth structures on $M^{2n}$ is in one-to-one correspondence with group $\overline{\Theta}_{2n}$. 
\end{theorem}
\begin{remark}\rm
\indent
\begin{itemize}
 \item[(1)] By Theorem \ref{tangent}, every closed smooth 4-connected 10-manifold has exactly six distinct smooth structures, namely, $\left \{[M\#\Sigma]~~|~~ \Sigma\in \overline{\Theta}_{10}\cong \mathbb{Z}_{6} \right \}.$
 \item[(2)]If $M^{2n}$ is $n$-parallelisable, almost parallelisable or $\pi$-manifold, then the stable tangential invariant $\chi$ of $M$ is zero \cite{Wal62}. Then by Theorem \ref{tangent}, we have the following : 
 \end{itemize}
\end{remark}
\begin{corollary}
 Let $n$ be any integer greater than $3$ such that $\Theta_{n}$ and $\Theta_{n+1}$ are trivial and $M^{2n}$ be a closed smooth $(n-1)$-connected $2n$-manifold.
If $M^{2n}$ is $n$-parallelisable, almost parallelisable or $\pi$-manifold, then the set of diffeomorphism classes of smooth structures on $M^{2n}$ is in one-to-one correspondence with group $\overline{\Theta}_{2n}$. 
\end{corollary}
\begin{definition}\cite{Kre99}
The normal $k$-type of a closed smooth manifold $M$ is the fibre homotopy type of a fibration $p:B\to BO$ such that the fibre of the map $p$ is connected
and its homotopy groups vanish in dimension $\geq k+1$, admitting a lift of the normal Gauss map $\nu_{M}:M\to BO$ to a map  $\bar{\nu}_{M}:M\to B$ such that $\bar{\nu}_{M}:M\to B$ is a $(k + 1)$-equivalence, i.e., the induced homomorphism  $\bar{\nu}_{M}:\pi_i(M)\to \pi_i(B)$ is an isomorphism
for $i\leq k$ and surjective for $i=k + 1$. We call such a lift a normal $k$-smoothing.
\end{definition}
\begin{theorem}\label{zerobordi}
Let $n=$5, 7 and let $M_0$ and $M_1$ be closed smooth $(n-1)$-connected $2n$-manifolds with the same Euler characteristic. Then
 \begin{itemize}
  \item [(i)] There is a homotopy sphere $\Sigma^{2n}\in \overline{\Theta}_{2n}$ such that $M_0$ and $M_1\#\Sigma^{2n}$ are diffeomorphic.
  \item[(ii)]Let $M^{2n}$ be a $(n-1)$-connected $2n$-manifold such that $[M]=0\in\Omega^{String}_{2n}$ and let $\Sigma$ be any exotic $2n$-sphere in $\overline{\Theta}_{2n}$. Then $M$ and $M\#\Sigma$ are not diffeomorphic.
 \end{itemize}
\end{theorem}
\begin{proof}
(i): $M_0$ and $M_1$ are  $(n-1)$-connected, and $n$ is $5$ or $7$; therefore, $\frac{p_1}{2}$ and the Stiefel-Whitney classes $\omega_2$ vanish. So, $M_0$ and $M_1$ are $BString$-manifolds. Let $\bar{\nu}_{M_j}:M_j\to BString$ be a lift of the normal gauss map $\nu_{M_j}:M_j\to BO$ in the fibration $p:BString=BO\left <8 \right >\to BO$, where $j=$0 and 1. Since $BString$ is 7-connected, $p_{\#}:\pi_{i}(BString)\to \pi_{i}(BO)$ is an isomorphism for all $i\geq 8$. This shows that $\bar{\nu}_{M_j}:M_j\to BString$ is an $n$-equivalence and hence the normal $(n-1)$-type of $M_0$ and $M_1$ is $p:BString\to BO$. We know that $\Omega^{String}_{2n}\cong \overline{\Theta}_{2n}$, where the group structure is given by connected sum \cite{Gia71}. This implies that there always exists  $\Sigma^{2n}\in \overline{\Theta}_{2n}$ such that $M_0$ and $M_1\#\Sigma^{2n}$ are $BString$-bordant. Since $M_0$ and $M_1\#\Sigma^{2n}$ have the same Euler characteristic, by \cite[Corollary 4]{Kre99}, $M_0$ and $M_1\#\Sigma^{2n}$ are diffeomorphic.\\
(ii): Since the image of the standard sphere under the isomorphism $\overline{\Theta}_{2n}\cong \Omega^{String}_{2n}$ represents the trivial element
in $\Omega^{String}_{2n}$, we have $[M^{2n}]\neq [M\#\Sigma]$ in $\Omega^{String}_{2n}$. This implies that $M$ and $M\#\Sigma$ are not $BString$-bordant. By obstruction theory, $M^{2n}$ has a unique string structure. This implies that $M$ and $M\#\Sigma$ are not diffeomorphic.
\end{proof}
\begin{theorem}\label{pi-mani}
 Let $M$ be a closed smooth 6-connected 14-dimensional $\pi$-manifold and $\Sigma$ is the exotic 14-sphere. Then $M\#\Sigma$ is not diffeomorphic to $M$. Thus, $I(M)=0$. Moreover, if $N$ is a closed smooth manifold homeomorphic to $M$, then $N$ is diffeomorphic to either $M$ or $M\#\Sigma$.
\end{theorem}
\begin{proof}
 It follows from results of Anderson, Brown and Peterson on spin cobordism \cite{ABP67} that the image of the natural homomorphism $\Omega_{14}^{framed}\to \Omega_{14}^{Spin}$ is 0 and $\Omega_{14}^{String}\cong \Omega_{14}^{Spin}\cong \mathbb{Z}_2$ \cite{Gia71}. This shows that $[M]=0\in \Omega_{14}^{String}$. Now by Theorem \ref{zerobordi} (ii), $M\#\Sigma$ is not diffeomorphic to $M$. If $N$ is a closed smooth manifold homeomorphic to $M$, then $N$ and $M$ have the same Euler characteristic. Then by Theorem \ref{zerobordi}(i), $N$ is diffeomorphic to either $M$ or $M\#\Sigma$. 
\end{proof}
\begin{remark}\rm
By the above Theorem \ref{pi-mani}, the set of diffeomorphism classes of smooth structures on a closed smooth $6$-connected $14$-dimensional $\pi$-manifold $M$ is  $$\left \{ [M], [M\#\Sigma] \right \}\cong \mathbb{Z}_2,$$ where $\Sigma$ is the exotic $14$-sphere. So, the number of distinct smooth structures on $M$ is $2$.
\end{remark}

\paragraph{ADDRESS:}
THEORETICAL STATISTICS AND MATHEMATICS UNIT, INDIAN STATISTICAL INSTITUTE, KOLKATA-700 108, INDIA.\\
E-mail : mathsramesh1984@gmail.com,\\ rameshkasilingam.iitb@gmail.com

\end{document}